\documentclass[12pt]{article}
\usepackage{amsmath}%
\usepackage{amsfonts}%
\usepackage{amssymb}%
\usepackage{amsthm}
\usepackage{hyperref}
\usepackage{graphicx}
\DeclareGraphicsExtensions{.eps}
\newtheorem{thm}{Theorem}[section]
\newtheorem{cor}[thm]{Corollary}

\newtheorem{lem}[thm]{Lemma}
\newtheorem{defn}[thm]{Definition}

\numberwithin{equation}{section}
\begin{document}
\begin{center}
\LARGE
Some Interesting Properties of the Riemann Zeta Function
\end{center}

\begin{center}
\Large
Johar M. Ashfaque
\end{center}
\tableofcontents
\section{Introduction}
Leonhard Euler lived from 1707 to 1783 and is, without a doubt, one of the most influential mathematicians of all time.
His work covered many areas of mathematics including algebra, trigonometry, graph theory, mechanics and, most relevantly, analysis.
\par Although Euler demonstrated an obvious genius when he was a child it took until 1735 for his talents to be fully recognised.
It was then that he solved what was known as the Basel problem, the problem set in 1644 and named after Euler's home town
\cite{Prime}. This problem asks for an exact expression of the limit of the equation
\begin{equation} \label{zeta2}
\sum^\infty_{n=1}\frac{1}{n^2}=1+\frac{1}{4}+\frac{1}{9}+\frac{1}{16}+...,
\end{equation}
which Euler calculated to be exactly equal to $\pi^2/6$. Going beyond this, he also calculated that
\begin{equation} \label{zeta4}
\sum^\infty_{n=1}\frac{1}{n^4}=1+\frac{1}{16}+\frac{1}{81}+\frac{1}{256}+...=\frac{\pi^4}{90},
\end{equation}
among other specific values of a series that later became known as the Riemann zeta function, which is classically defined in the following way.
\begin{defn}\label{zeta} For $\Re(s)>1$, the Riemann zeta function is defined as
$$\zeta(s)=\sum^\infty_{n=0}\frac{1}{n^s}=1+\frac{1}{2^s}+\frac{1}{3^s}+\frac{1}{4^s}+...$$
\end{defn}
\noindent This allows us to write the sums in the above equations (\ref{zeta2}) and (\ref{zeta4}) simply as $\zeta(2)$ and $\zeta(4)$ respectively.
A few years later Euler constructed a general proof that gave exact values for all $\zeta(2n)$ for $n\in\mathbb{N}$. These were the first instances
of the special values of the zeta function and are still amongst the most interesting. However, when they were discovered, it was still unfortunately
the case that analysis of $\zeta(s)$ was restricted only to the real numbers.
It wasn't until the work of Bernhard Riemann that the zeta function was to be fully extended to all of the complex numbers by the process of
analytic continuation and it is for this reason that the function is commonly referred to as the Riemann zeta function. From this, we are able to
calculate more special values of the zeta function and understand its nature more completely.
\par We will be discussing some classical values of the zeta function as well as some more modern ones and will require no more than an
undergraduate understanding of analysis (especially Taylor series of common functions) and complex numbers. The only tool that we will be using
extensively in addition to these will be the `Big O' notation, that we shall now define.
\begin{defn} \label{bigoh}
We can say that $f(x)=g(x)+O(h(x))$ as $x \rightarrow k$ if there exists a constant $C> 0$ such that $|f(x)-g(x)|\leq C| h(x)|$ for when $x$ is close enough to $k$.
\end{defn}
\noindent This may seem a little alien at first and it is true that, to the uninitiated, it can take a little while to digest.
However, its use is more simple than its definition would suggest and so we will move on to more important matters.

\section{The Euler Product Formula for $\zeta(s)$}
We now define the convergence criteria for infinite products.
\begin{thm}\label{prodcon}
If $\sum|a_n|$ converges then $\prod(1+a_n)$ converges.
\end{thm}
\noindent Although we will not give a complete proof here, one is referred to \cite{Whit}. However, it is worth noting that
$$\prod^m_{n=1}(1+a_n)=\exp\left\{\sum^m_{n=1}\ln(1+a_n)\right\}.$$
From this it can be seen that, for the product to converge, it is sufficient that the summation on the right hand side converges.
The summation $\sum\ln(1+a_n)$ converges absolutely if $\sum|a_n|$ converges. This is the conceptual idea that the proof is rooted in.
\par Now that we have this tool, we can prove the famous Euler Product formula for $\zeta(s).$ This relation makes use of the Fundamental Theorem of Arithmetic which says that every integer can be expressed as a unique product of primes. We will not prove it here but the interested reader can consult \cite{Chahal} for a proof.
\begin{thm}\label{eulerprod}
Let $p$ denote the prime numbers. For $\Re(s)>1,$
$$\zeta(s)=\prod_{p}^\infty(1-p^{-s})^{-1}.$$
\end{thm}
\begin{proof} Observe that
$$\frac{\zeta(s)}{2^s}=\frac{1}{2^s}+\frac{1}{4^s}+\frac{1}{6^s}+\frac{1}{8^s}+\frac{1}{10^s}+...$$
Then, we can make the subtraction
\begin{align*}
\zeta(s)\left(1-\frac{1}{2^s}\right)&=1+\frac{1}{2^s}+\frac{1}{3^s}+\frac{1}{4^s}+\frac{1}{5^s}+...\\
&-\frac{1}{2^s}+\frac{1}{4^s}+\frac{1}{6^s}+\frac{1}{8^s}+\frac{1}{10^s}+...\\
&=1+\frac{1}{3^s}+\frac{1}{5^s}+\frac{1}{7^s}+\frac{1}{9^s}+...
\end{align*}
Clearly, this has just removed any terms from $\zeta(s)$ that have a factor of $2^{-s}.$ We can then take this a step further to see that
$$\zeta(s)\left(1-\frac{1}{2^s}\right)\left(1-\frac{1}{3^s}\right)=1+\frac{1}{5^s}+\frac{1}{7^s}+\frac{1}{11^s}+\frac{1}{13^s}+...$$
If we continue this process of siphoning off primes we can see that, by the Fundamental Theorem of Arithmetic,
$$\zeta(s)\prod_p(1-p^{-s})=1,$$
which requires only a simple rearrangement to see that
$$\zeta(s)=\prod_p(1-p^{-s})^{-1}.$$
Note that, for $\Re(s)>1$ this converges because
$$\sum_p p^{-\Re(s)}<\sum_n n^{-\Re(s)},$$
which converges. This completes the proof
\end{proof}

\section{The Bernoulli Numbers}
We will now move on to the study of Bernoulli numbers, a sequence of rational numbers that pop up frequently when considering the zeta function. We are interested in them because they are intimately related to some special values of the zeta function and are present in some rather remarkable identities.
\par We already have an understanding of Taylor series and the analytic power that they provide and so we can now begin with the definition of the Bernoulli numbers. This section will follow Chapter 6 in \cite{Stopple}.
\begin{defn} \label{bern}
The Bernoulli Numbers $B_n$ are defined to be the coefficients in the series expansion
$$\frac{x}{e^x-1}=\sum_{n=0}^\infty\frac{B_n x^n}{n!}.$$
\end{defn}
\noindent It is a result from complex analysis that this series converges for $|x|<2\pi$ but, other than this, we cannot gain much of an understanding from the implicit definition. Please note also, that although the left hand side would appear to become infinite at $x=0$, it does not.
\begin{cor} \label{bernlaw}
We can calculate the Bernoulli numbers by the recursion formula
$$0=\sum_{j=0}^{k-1}{k \choose j}B_j,$$
where $B_0=1$.
\end{cor}
\begin{proof}
Let us first replace $e^x-1$ with its Taylor series to see that
$$x=\left(\left(\sum_{j=0}^\infty\frac{x^j}{j}\right)-1\right)\sum_{n=0}^\infty \frac{B_n x^n}{n!}=\sum_{j=1}^\infty\frac{x^j}{j!}\sum_{n=0}^\infty \frac{B_n x^n}{n!}.$$
If we compare coefficients of powers of $x$ we can clearly see that, except for $x^1$,
$$x^k: \hspace{2mm}0=\frac{B_0}{k!}+\frac{B_1}{(k-1)!}+\frac{B_2}{2!(k-2)!}+...+\frac{B_{k-2}}{2!(k-2)!}+\frac{B_{k-1}}{(k-1)!}.$$
Hence
$$0=\sum^{k-1}_{j=0}\frac{B_j}{(k-j)!j!}=\frac{1}{k!}\sum_{j=0}^{k-1}{k \choose j}B_j.$$
Note that the inverse $k!$ term is irrelevant to the recursion formula. This completes the proof.
\end{proof}
\noindent The first few Bernoulli numbers are therefore
$$B_0=1,\hspace{2mm} B_1=-1/2,\hspace{2mm} B_2=1/6,\hspace{2mm} B_3=0,$$
$$B_4=-1/30, \hspace{2mm} B_5=0, \hspace{2mm} B_6=1/42, \hspace{2mm} B_7=0.$$

\begin{lem} \label{oddbern}
The values of the odd Bernoulli numbers (except $B_1$) are zero
\end{lem}
\begin{proof}  As we know the values of $B_0$ and $B_1$, we can remove the first two terms from Definition \ref{bern} and rearrange to get
$$\frac{x}{e^x-1}+\frac{x}{2}=1+\sum^\infty_{n=2}\frac{B_n x^n}{n!},$$
which then simplifies to give
$$\frac{x}{2}\left(\frac{e^x+1}{e^x-1}\right)=1+\sum^\infty_{n=2}\frac{B_n x^n}{n!}.$$
We can then multiply both the numerator and denominator of the left hand side by $\exp(-x/2)$ to get
\begin{equation} \label{berndef2}
\frac{x}{2}\left(\frac{e^{x/2}+e^{-x/2}}{e^{x/2}-e^{-x/2}}\right)=1+\sum^\infty_{n=2}\frac{B_n x^n}{n!}.
\end{equation}
\noindent By substituting $x\rightarrow-x$ into the left hand side of this equation we can see that it is an even function and hence
invariate under this transformation. Hence, as the odd Bernoulli numbers multiply odd powers of $x$, the right hand side can only be invariant
under the same transformation if the value of the odd coefficients are all zero.
\end{proof}

\subsection{Relationship to the Zeta Function}
As, we have already dicussed, Euler found a way of calculating exact values of $\zeta(2n)$ for $n\in\mathbb{N}.$ He did this using the properties of the Bernoulli numbers, although he originally did it using the infinite product for the sine function. The relationship between the zeta function and Bernoulli numbers is not obvious but the proof of it is quite satisfying.
\begin{thm} \label{zetabern}
For $n\in\mathbb{N}$,
$$\zeta(2n)=(-1)^{n-1}\frac{(2\pi)^{2n}B_{2n}}{2(2n)!}.$$
\end{thm}
\noindent To prove this theorem, we will be using the original proof attributed to Euler and reproduced in \cite{Stopple}. This will be done by finding two seperate expressions for $z\cot(z)$ and then comparing them. We will be using a real analytic proof which is slightly longer than a complex analytic proof, an example of which can be found in \cite{Whit}.

\begin{lem} \label{bernlem1}
The function $z\cot(z)$ has the Taylor expansion
\begin{equation} \label{eqbernlem1}
z\cot(z)=1+\sum_{n=1}^\infty(-4)^n\frac{B_{2n}z^{2n}}{(2n)!}.
\end{equation}
\end{lem}
\begin{proof}
Substitute $x=2iz$ into equation (\ref{berndef2}) and observe that, because the odd Bernoulli numbers are zero, we can write this as
$$\frac{x}{2}\left(\frac{e^{x/2}+e^{-x/2}}{e^{x/2}-e^{-x/2}}\right)=iz\frac{e^{iz}+e^{-iz}}{e^{iz}-e^{-iz}}=1+\sum_{n=1}^\infty(-4)^n\frac{B_{2n}z^{2n}}{(2n)!}.$$
Noting that the left hand side is equal to $z\cot(z)$ completes the proof
\end{proof}
\begin{lem} \label{bernlem2}
The function $\cot(z)$ can be written as
\begin{equation} \label{eqbernlem2}
\cot(z)=\frac{\cot(z/2^n)}{2^n}-\frac{\tan(z/2^n)}{2^n}+\frac{1}{2^n}\sum_{j=1}^{2^{n-1}-1}\left(\cot\left(\frac{z+j\pi}{2^n}\right)+\cot\left(\frac{z-j\pi}{2^n}\right)\right).
\end{equation}
\end{lem}
\begin{proof}
Recall that $2\cot(2z)=\cot(z)+\cot(z+\pi /2)$. If we continually iterate this formula we will find that
$$\cot(z)=\frac{1}{2^n}\sum_{j=0}^{2^n-1}\cot\left(\frac{z+j\pi}{2^n}\right),$$
which can be proved by induction. Removing the $j=0$ and $j=2^{n-1}$ terms and recalling that $\cot(z+\pi /2)=-\tan(z)$ gives us
\begin{eqnarray*}\cot(z)&=&\frac{\cot(z/2^n)}{2^n}-\frac{\tan(z/2^n)}{2^n}\\&+&\frac{1}{2^n}
\left(\left(\sum_{j=1}^{2^{n-1}-1}\cot\left(\frac{z+j\pi}{2^n}\right)\right)+
\left(\sum_{j=2^{n-1}+1}^{2^n -1}\cot\left(\frac{z+j\pi}{2^n}\right)\right)\right).\\ \end{eqnarray*}
All we have to do now is observe that, as $\cot(z+\pi)=\cot(z)$, we can say that
$$\sum^{2^n-1}_{j=2^{n-1}+1}\cot\left(\frac{z+j\pi}{2^n}\right)=\sum_{j=1}^{2^{n-1}-1}\cot\left(\frac{z-j\pi}{2^n}\right),$$
which completes the proof.
\end{proof}
\begin{lem}\label{bernlem3}
The function $z\cot(z)$ can therefore be expressed as
\begin{equation} \label{eqbernlem3}
z\cot(z)=1-2\sum_{j=1}^\infty\frac{z^2}{j^2\pi^2-z^2}.
\end{equation}
\end{lem}
\begin{proof}
In order to obtain this, we first multiply both sides of equation (\ref{eqbernlem2}) by $z$ to get
\begin{equation} \label{bigzcotz}
z\cot(z)=\frac{z}{2^n}\cot(z/2^n)-\frac{z}{2^n}\tan(z/2^n)+\sum_{j=1}^{2^{n-1}-1}\frac{z}{2^n}\left(\cot\left(\frac{z+j\pi}{2^n}\right)+\cot\left(\frac{z-j\pi}{2^n}\right)\right).
\end{equation}
\noindent Let us now take the limit of the right hand side as $n$ tends to infinity. First recall that the Taylor series for $x\cot(x)$ and $x\tan(x)$ can be respectively expressed as
$$x\cot(x)=1+O(x^2)$$
and
$$x\tan(x)=x^2+O(x^4).$$
\noindent Hence, if we substitute $x=z/2^n$ into both of these we can see that
\begin{equation}\label{cottay}
\lim_{n\rightarrow\infty}\left[\frac{z}{2^n}\cot\left(\frac{z}{2^n}\right)\right]=1
\end{equation}
and
\begin{equation}\label{tantay}
\lim_{n\rightarrow\infty}\left[\frac{z}{2^n}\tan\left(\frac{z}{2^n}\right)\right]=0.
\end{equation}
\noindent Now we have dealt with the expressions outside the summation and so we need to consider the ones inside. To make things slightly easier for the moment, let us consider both of the expressions at the same time. Using Taylor series again, we can see that
\begin{equation} \label{cottay2}
\frac{z}{2^n}\cot\left(\frac{z\pm j\pi}{2^n}\right)=\frac{z}{z\pm j\pi}+O(4^{-n}).
\end{equation}
\noindent Substituting equations (\ref{cottay}), (\ref{tantay}) and (\ref{cottay2}) into the right hand side of equation (\ref{bigzcotz}) gives that
$$z\cot(z)=1-\lim_{n\rightarrow\infty}\sum_{j=1}^{2^{n-1}-1}\left[\frac{z}{z+j\pi}+\frac{z}{z-j\pi}+O(4^{-n})\right],$$
which simplifies a little to give
$$z\cot(z)=1-2\sum^\infty_{j=1}\frac{z^2}{j^2\pi^2-z^2}-\lim_{n\rightarrow\infty}\sum_{j=1}^{2^{n-1}-1}O(4^{-n}).$$
By Definition \ref{bigoh}, it can be seen that
$$|\sum_{j=1}^{2^{n-1}-1}O(4^{-n})| \leq C(2^{n-1}-1)4^{-n},$$
which clearly converges to zero as $n\rightarrow \infty$, thus completing the proof.
\end{proof}
\begin{lem} \label{bernlem4}
For $|z| < \pi$, $z\cot(z)$ has the expansion
\begin{equation} \label{eqbernlem4}
z\cot(z)=1-2\sum_{j=1}^\infty\zeta(2n)\frac{z^{2n}}{\pi^{2n}}.
\end{equation}
\end{lem}
\begin{proof}
Take the summand of equation (\ref{bernlem3}) and multiply both the numerator and denominator by $(j\pi)^{-2}$ to obtain
$$z\cot(z)=1-2\sum_{j=1}^\infty\frac{(z/j\pi)^2}{1-(z/j\pi)^2}.$$
But, we can note that the summand can be expanded as an infinite geometric series. Hence we can write this as
$$z\cot(z)=1-2\sum_{j=1}^\infty\sum_{n=1}^\infty\left(\frac{z}{j\pi}\right)^{2n},$$
which
$$=1-2\sum_{n=1}^\infty\left(\frac{z}{\pi}\right)^{2n}\zeta(2n)$$
as long as the geometric series converges (i.e. $|z| < \pi$). Note that exchanging the summations in such a way is valid as both of the series are absolutely convergent.
\end{proof}
\noindent Now, we can complete the proof of the Theorem \ref{zetabern} by equating equations (\ref{eqbernlem1}) and (\ref{eqbernlem4}) to see that
$$1+\sum_{n=2}^\infty=(-4)^n\frac{B_{2n}z^{2n}}{(2n)!}=1-2\sum_{n=1}^\infty\left(\frac{z}{\pi}\right)^{2n}\zeta(2n).$$
If we then strip away the 1 terms and remove the summations we obtain the identity
$$(-4)^n\frac{B_{2n}z^{2n}}{(2n)!}=-2\frac{z^{2n}}{\pi^{2n}}\zeta(2n),$$
which rearranges to complete the proof of Theorem \ref{zetabern} as required.

Now that have proven this beautiful formula (thanks again, Euler) we can use it to calculate the sums of positive even values of the zeta function. First, let us rewrite the result of Theorem \ref{zetabern} to give us
$$\zeta(2n)=\frac{(2\pi)^{2n}|B_{2n}|}{2(2n)!}.$$
From this, we can easily calculate specific values such as
$$\zeta(6)=\sum_{n=1}^\infty\frac{1}{j^6}=\frac{(2\pi)^6|B_{6}|}{2(6)!}=\frac{\pi^{6}}{945},$$
$$\zeta(8)=\sum_{n=1}^\infty\frac{1}{j^8}=\frac{(2\pi)^8|B_{8}|}{2(8)!}=\frac{\pi^{8}}{9450},$$
etc. This is a beautiful formula and it is unfortunate that no similar formula has been discovered for $\zeta(2n+1)$.
\par However, that's not to say that there aren't interesting results regarding these values! There have been recent results concerning the values of the zeta function for odd integers. For example, Ap\'{e}ry's proof of the irrationality of
$\zeta(3)$ in 1979 or Matilde Lal\'{i}n's integral representations of $\zeta(3)$ and $\zeta(5)$ by the use of Mahler Measures.
\par  Mercifully, special values for $\zeta(-2n)$ and $\zeta(-2n+1)$ have been found, the latter of which also involves Bernoulli numbers! It is to our good fortune that we will have to take a whirlwind tour through some of the properties of the Gamma function in order to get to them.

\section{The Gamma Function}
The Gamma function is closely related to the zeta function and as such warrants some exploration of its more relevant properties. We will only be discussing a few of the qualities of this function but the reader should note that it has many applications in statistics (Gamma and Beta distributions) and orthoganal functions (Bessel functions). It was first investigated by Euler when he was considering factorials of non-integer numbers and it was through this study that many of its classical properties were established.
\par This section will follow Chapter 8 in \cite{Stopple} with sprinklings from \cite{Gamma}. First, let us begin with a definition from Euler.
\begin{defn} \label{gamma}
For $\Re(s)> 0$ we define $\Gamma(s)$ as
$$\Gamma(s)=\int^\infty_0 t^{s-1}e^{-t}dt.$$
\end{defn}
\noindent Now, this function initially looks rather daunting and irrelavent. We will see, however, that is does have many fascinating properties.
Among the most basic are the following two identities ...
\begin{cor} \label{sgammas}
$\Gamma(s)$ has the recursive property
\begin{equation}\label{eqsgammas}
\Gamma(s+1)=s\Gamma(s).
\end{equation}
\end{cor}
\begin{proof} We can prove this by performing a basic integration by parts on the gamma function. Note that
\begin{eqnarray*}
\Gamma(s+1)&=&\int^\infty_0 t^s e^{-t}dt=\left[-t^s e^{-t}\right]^\infty_0+s\int^\infty_0 t^{s-1}e^{-t}dt\\
&=&[0]+s\Gamma(s)\\
\end{eqnarray*}
as required.
\end{proof}

\begin{cor}\label{gammfac}
For $n\in\mathbb{N}$
\begin{equation} \label{eqgammafac}
\Gamma(n+1)=n!
\end{equation}
\end{cor}
\begin{proof}
Just consider Corollary \ref{sgammas} and iterate to complete the proof.
\end{proof}
\noindent It is this property that really begins to tell us something interesting about the Gamma function. Now that we know that the function calculates factorials for integer values, we can use it to `fill in' non-integer values, which is the reason why Euler introduced it.
\newline \newline
\noindent \textit{Remark.} We can use the fact that $\Gamma(s)=\Gamma(s+1)/s$ to see that, as $s$ tends to 0, $\Gamma(s)\rightarrow\infty$. We can also use this recursive relation to prove that the Gamma function has poles at all of the negative integers. However, the more beautiful proof of this is to come in Section 6.

\begin{lem} \label{polar}
We can calculate that
$$\Gamma(3/2)=\frac{\sqrt{\pi}}{2}.$$
\end{lem}
\begin{proof}
We observe that
\begin{equation} \label{gamma32}
\Gamma(3/2)=\int^\infty_0\exp(-x^2)dx.
\end{equation}
\noindent This is an important result that is the foundation of the normal distribution and it is also easily calculable. We do this by first considering the double integral
$$I=\int^\infty_{-\infty}\int^\infty_{-\infty}\exp(-x^2-y^2)dxdy.$$
If we switch to polar co-ordinates using the change of variables $x=r\cos(\theta)$, $y=r\sin(\theta).$ Noting that $dydx=rd\theta dr$, we have
$$I=\int_0^{2\pi}\int_0^\infty r \exp(-r^2)drd\theta=\pi\int^\infty_02r\exp(-r^2)dr=\pi[\exp(-r^2)]^\infty_0=\pi.$$
We can then seperate the original integral into two seperate integrals to obtain
$$I=\left(\int^\infty_{-\infty}\exp{-x^2}dx\right)\left(\int^\infty_{-\infty}\exp{-y^2}dy\right)=\pi.$$
Noting that the two integrals are identical and are also both even functions, we can see that integrating one of them from zero to infinity completes the proof as required
\end{proof}
\begin{cor} \label{gamma2nplus1}
Consider $(n)!_2=n(n-2)(n-4)...,$, which terminates at $1$ or $2$ depending on whether $n$ is odd or even respectively. Then for $n\in\mathbb{N},$
$$\Gamma\left(\frac{2n+1}{2}\right)=\frac{\sqrt{\pi}(2n-1)!_2}{2^n}.$$
\end{cor}
\begin{proof}  We will prove this by induction. Consider that
$$\Gamma\left(\frac{2n+3}{2}\right)=\Gamma\left(\frac{2(n+1)+1}{2}\right)=\frac{(2n+1)}{2}.\frac{(2n-1)!_2 \sqrt{\pi}}{2^n}=\frac{(2n+1)!_2 \sqrt{\pi}}{2^{n+1}}.$$
Noting that the leftmost and rightmost equalities are equal by definition completes the proof.
\end{proof}
\noindent \textit{Remark.} We can use this relationship $\Gamma(1+1/s)=(1/s)\Gamma(s)$ to see, for example, that
$$\Gamma(5/2)=\frac{3\sqrt{\pi}}{4}, \hspace{2mm} \Gamma(7/2)=\frac{15\sqrt{\pi}}{8},$$
etc.
\begin{cor} \label{factminushalf}
We can computer the `factorial' of $-1/2$ as
$$\Gamma(-1/2)=-2\sqrt{\pi}.$$
\end{cor}
\begin{proof} We can rework Corollary \ref{sgammas} to show that
$$\Gamma(1/s)=\frac{1-s}{s}\Gamma\left(\frac{1-s}{s}\right),$$
from which the corollary can easily be proven.
\end{proof}

\subsection{The Euler Reflection Formula}
This chapter will use a slightly different definition of the Gamma function and will follow source \cite{Gamma}. First let us consider the definition of the very important Euler constant $\gamma$.
\begin{defn} \label{eulerconstant}
Euler's constant $\gamma$ is defined as
$$\gamma=\lim_{m\rightarrow\infty}\left(1+\frac{1}{2}+\frac{1}{3}+..+\frac{1}{m}-\log(m)\right).$$
\end{defn}
\noindent We will then use Gauss' definition for the Gamma function which can be written as follows ...
\begin{defn} \label{gammagauss}
For $s>0$ we can define
$$\Gamma_h(s)=\frac{h!h^s}{s(s+1)..(s+p)}=\frac{h^s}{s(1+s)(1+s/2)...(1+s/h)}$$
and
$$\Gamma(s)=\lim_{h\rightarrow\infty}\Gamma_h(s).$$
\end{defn}
\noindent This does not seem immediately obvious but the relationship is true and is proven for $\Re(s)>0$ in \cite{Whit}. So now that we have these definitions we can work on a well known theorem.
\begin{thm} \label{gammaprod}
The Gamma function can be written as the following infinite product;
$$\frac{1}{\Gamma(s)}=se^{\gamma s}\prod_{n=1}^\infty \left(1+\frac{s}{n}\right)e^{-s/n}.$$
\end{thm}
\begin{proof}
Before we start with the derivation, let us note that the infinite product is convergent because the exponential term forces it. Now that we have cleared that from our conscience, we will begin by using Definition \ref{gammagauss} and say that
$$\Gamma_h(s)=\frac{h^s}{s(1+s)(1+s/2)...(1+s/h)}.$$
Now we can also see that
\begin{eqnarray*}
h^s&=&\exp\left(s\log(h)\right)\\
&=&\exp\left\{s\left(\log(h)-1-\frac{1}{2}-...-\frac{1}{h}\right)\right\}\exp\left\{s\left(1+\frac{1}{2}+...+\frac{1}{h}\right)\right\}.\\
\end{eqnarray*}
We can then observe that
$$\Gamma_h(s)=\frac{1}{s}\frac{e^s}{1+s}\frac{e^{s/2}}{1+s/2}...\frac{e^{s/h}}{1+s/h}\exp\left\{s\left(\log(h)-1-\frac{1}{2}-...-\frac{1}{h}\right)\right\},$$
which we can write as the product
$$\Gamma_h(s)=\frac{1}{s}\exp\left\{s\left(\log(h)-1-\frac{1}{2}-...-\frac{1}{h}\right)\right\}\prod_{n=1}^h\frac{1}{1+s/n}e^{s/n}.$$
All we need to do now is to take the limit of this as $h$ tends to infinity and use Definition \ref{eulerconstant} to prove the theorem as required.
\end{proof}
\noindent This theorem is very interesting as it allows us to prove two really quite beautiful identities, known as the Euler Reflection formulae. But before we do this, we are going to need another way of dealing with the sine function. It should be noted that the method of approach that we are going to use is not completely rigorous. However, it can be proven rigorously using the Weierstrass Factorisation Theorem - a discussion of which can be found in \cite{Whit}.
\begin{thm} \label{eulersin}
The sine function has the infinite product
$$\sin(\pi s)=\pi s\prod_{n=1}^\infty\left(1-\frac{s^2}{n^2}\right).$$
\end{thm}
\begin{thm} \label{eulerreflection}
The Gamma function has the following reflective relation -
$$\frac{1}{\Gamma(s)\Gamma(1-s)}=\frac{\sin(\pi s)}{\pi}.$$
\end{thm}
\begin{proof}
We can use Theorem \ref{gammaprod} to see that
$$\frac{1}{\Gamma(s)\Gamma(-s)}=-s^2e^{\gamma s - \gamma s}\prod_{n=1}^\infty e^{s/n-s/n}\left(\frac{n+s}{n}\right)\left(\frac{n-s}{n}\right)=-s^2\prod_{n=1}^\infty\frac{n^2-s^2}{n^2}.$$
We can then use Corollary \ref{sgammas} to show that, as $\Gamma(1-s)=-s\Gamma(-s)$,
$$\frac{1}{\Gamma(s)\Gamma(1-s)}=s\prod_{n=1}^\infty\frac{n^2-s^2}{n^2}.$$
Comparing this to Theorem \ref{eulersin} then completes the proof as required.
\end{proof}
\begin{cor} \label{eulerreflection2}
The Gamma function also has the reflectional formula
$$\frac{1}{\Gamma(s)\Gamma(-s)}=-\frac{s \sin(\pi s)}{\pi}.$$
\end{cor}
\begin{proof}
This can easily be shown using a slight variation of the previous proof. However, an alternate proof can be constructed by considering Theorem \ref{eulerreflection} and Corollary \ref{sgammas}.
\end{proof}

\end{document}